\documentclass[onecolumn]{IEEEtran}
\usepackage{amsmath,amsfonts,amssymb,euscript, graphicx, 
epsfig,
enumerate,float,afterpage, subfigure, ifthen}%

\usepackage{url}
\usepackage{hyperref}

\newtheorem{thm}{Theorem}
\newtheorem{cor}{Corollary}
\newtheorem{lem}{Lemma}
\newtheorem{defn}{Definition}
\newtheorem{assumption}{Assumption}
\newtheorem{prop}{Proposition}

\newcommand{\expect}[1]{\mathbb{E}\left[#1\right]}

\newcommand{\script}[1]{{{\cal{#1} }}}

\newcommand{\transpose}{\top}

\begin{document}

\title
  {Random Variables with Measurability Constraints with Application to Opportunistic Scheduling}
\author{Michael J. Neely \\ University of Southern California\\ \url{https://viterbi-web.usc.edu/~mjneely/}\\%$\vspace{-.4in}$%
}

\markboth{}{Neely}

\maketitle

\begin{abstract} 
This paper proves a representation theorem regarding sequences of 
random elements that take values in a Borel space and are measurable with respect to the sigma algebra generated by an 
arbitrary union of sigma algebras. This, together with a related representation theorem of Kallenberg, is used to characterize the set of multidimensional decision vectors in a discrete time stochastic control problem with measurability and causality constraints, including opportunistic scheduling problems for time-varying communication networks. A network capacity theorem for these systems is refined, without requiring an implicit and arbitrarily complex extension of the state space, by introducing two measurability assumptions and using a theory of constructible sets. An example that makes use of well known pathologies in descriptive set theory is given to show a nonmeasurable scheduling scheme can outperform all measurable scheduling schemes. 
\end{abstract} 

\section{Introduction}

Let $(\Omega, \script{F})$ and $(\Gamma, \script{G})$ be two measurable spaces. 
Let $X:\Omega \rightarrow \mathbb{R}$ be a Borel measurable function and let $Y:\Omega\rightarrow \Gamma$
be a measurable function. The Doob-Dynkin lemma states that $X$ is $\sigma(Y)$-measurable if and only if 
$X=h(Y)$ for some Borel measurable function $h:\Gamma\rightarrow\mathbb{R}$  \cite{doob-stochastic-processes}\cite{kallenberg}\cite{williams-martingale}. 
Suppose we know only that 
$X$ is $\sigma(\script{H}_1 \cup \script{H}_2)$-measurable, where $\script{H}_1\subseteq \script{F}$ and $\script{H}_2\subseteq\script{F}$ are
two subsigma algebras on $\Omega$.   Is it necessarily true that $X=h(Y_1, Y_2)$ for
some Borel measurable function $h:[0,1]^2\rightarrow\mathbb{R}$ and some Borel measurable functions $Y_i:\Omega\rightarrow [0,1]$ such that $Y_i$ is $\script{H}_i$-measurable for each  $i \in \{1,2\}$?  

This question motivates the more general question of characterizing all sequences of Borel measurable
functions that satisfy certain \emph{measurability constraints}.   Fix $K$ 
as a nonempty set that is finite or countably infinite. For each $k \in K$ let $X_k:\Omega\rightarrow\mathbb{R}$ be a function.
Let $J$ be a nonempty set with arbitrarily large cardinality.  For  each $j \in J$, let $\script{H}_j\subseteq \script{F}$ be a given subsigma algebra on $\Omega$. We characterize all $(X_k)_{k \in K}$ that satisfy 
\begin{equation} \label{eq:intro-constraints} 
X_k \mbox{ is $\sigma(\cup_{j \in J_k} \script{H}_j)$-measurable} \quad \forall k \in K 
\end{equation} 
where $J_k$ are given nonempty sets that satisfy $J_k\subseteq J$ for all $k \in K$. 
The first result is that $(X_k)_{k\in K}$ satisfies \eqref{eq:intro-constraints} if and only 
if  
\begin{equation} \label{eq:result} 
X_k= h_k((Y_j)_{j \in \tilde{J}_k}) \quad \forall k \in K
\end{equation} 
for some Borel measurable functions $Y_j:\Omega\rightarrow [0,1]$ that are $\script{H}_j$-measurable for each $j \in J$,   some countable subsets $\tilde{J}_k\subseteq J_k$, and some 
Borel measurable functions $h_k:[0,1]^{\tilde{J}_k}\rightarrow\mathbb{R}$ for each $k \in K$. Measurability of each function $h_k$ is with respect to the product sigma algebra on $[0,1]^{\tilde{J}_k}$. 
Observe that each $X_k$ in \eqref{eq:result} draws from the the \emph{same} collection of functions $(Y_j)_{j\in J}$ (rather than defining variables $Y_{j,k}$ separately for each $k$). In particular, a single function $Y_j$ can be used to represent the influence of the sigma algebra $\script{H}_j$ whenever that influence is required.   A special case of this result gives an affirmative answer to the question posed
in the first paragraph.  The result \eqref{eq:intro-constraints}-\eqref{eq:result} immediately generalizes to allow $X_k$ to be a random element of any Borel space, such as the space $(\mathbb{R}^m, \script{B}(\mathbb{R}^m))$ for some positive integer $m$. 

\subsection{Applications to stochastic control} 

The measurability constraints \eqref{eq:intro-constraints} have applications to stochastic control. 
For example, consider a discrete time system that operates over 
time slots $k \in \{1, 2, 3, \ldots\}$ according to some probability triplet $(\Omega, \script{F}, P)$. Let $S_k:\Omega\rightarrow\Omega_S$ be the \emph{system state} that can be observed at time $k$, which is a random element associated with some  measurable space $(\Omega_S, \script{F}_S)$ with arbitrary 
structure. Every step $k$ the system controller observes $S_k$ 
and chooses a \emph{decision} that affects a vector of attributes $X_k \in \mathbb{R}^m$, 
where $m$ is some fixed positive integer. The vector $X_k$  is required to satisfy the following \emph{system constraints} 
\begin{equation} \label{eq:system-constraints}
X_k \in C(S_k) \quad \forall k \in \{1, 2, 3, \ldots\}
\end{equation} 
where  $C:\Gamma\rightarrow Pow(\mathbb{R}^m)$ is a given set-valued function that maps the observed state $S_k$ to a subset of $\mathbb{R}^m$ that consists of all decision options for $X_k$ (where $Pow(\mathbb{R}^m)$ denotes the power set of $\mathbb{R}^m$).  The next state $S_{k+1}$ can be influenced by the prior states and decisions according to some model supported by the probability space, such as a Markov chain model.  The $m$ components of $X_k$ can represent rewards, prices, power expenditures, and so on, associated with time slot $k$, and can also include values that affect the next state.  This work is  motivated by the application of \emph{opportunistic scheduling}, 
where $(S_k)_{k=1}^{\infty}$ are independent and identically distributed (i.i.d.) random channel states that are sequentially observed in a wireless communication system at the start of each slot $k$, $X_k$ is a vector of transmission rates over $m$ different channels, and $C(S_k)$ is the set of all possible transmission rate vectors that can be supported on slot $k$ when the observed channel state is $S_k$.  

Consider \emph{causal and measurable} decision policies that are constrained to make decisions that yield valid random variables and are based only on observations of the past.  Assume the decisions can be \emph{stochastic}, so they can be informed by an external source of randomness that is represented by some 
sigma algebra $\script{G} \subseteq \script{F}$ on $\Omega$. For example, $\script{G}$ might be the sigma algebra generated by an infinite sequence of i.i.d. random elements in some arbitrary measurable space and whose values are selected by an independent computing device at time $0$ (before any control decisions are made).  
Then we require: 
%\begin{align*}
%&X_1 \mbox{ is $\sigma(\sigma(S_1) \cup \script{G})$-measurable} \\
%&X_2 \mbox{ is $\sigma(\sigma(S_1) \cup \sigma(S_2) \cup \script{G})$-measurable}
%\end{align*}
%and so on, so that 
\begin{equation} \label{eq:causal-measurable} 
X_k \mbox{ is $\sigma(\sigma(S_1)\cup \cdots \cup\sigma(S_k) \cup \script{G})$-measurable} \quad \forall k \in \{1, 2, 3, \ldots\}
\end{equation} 
where $\sigma(S_i)$ is the sigma algebra generated by the random element $S_i$.  Under any such causal and measurable decision policy the result \eqref{eq:result}  implies 
\begin{equation} \label{eq:hard-boiled}
 X_k = h_k(Y_1, \ldots, Y_k, R) \quad \forall k \in \{1, 2, 3, \ldots\}
 \end{equation} 
for some Borel measurable functions $h_k$, some $\script{G}$-measurable random variable $R$ that takes values in $[0,1]$, and some $\sigma(S_i)$-measurable random variables $Y_i$ that take values in $[0,1]$. It is interesting that the \emph{same} random variables $R, \{Y_i\}_{i=1}^{\infty}$ can be used to construct $X_k$ for all time steps $k$. In particular: 
 \begin{itemize} 
 \item While the observed random elements $S_i$ are associated with an arbitrarily complex measurable space $(\Omega_S, \script{F}_S)$ where $\Omega_S$ has arbitrary cardinality,  it suffices to boil these random elements down to real-valued random variables  $Y_i:\Omega\rightarrow [0,1]$ where each $Y_i$ is a measurable function of $S_i$.
 \item While the external source of randomness is from an arbitrarily complex sigma algebra $\script{G}$ on $\Omega$, it suffices to boil it down to a single draw of a random variable $R:\Omega\rightarrow [0,1]$ that is $\script{G}$-measurable. 
 \end{itemize} 
 
 The constraint $X_k \in C(S_k)$ seems to require knowledge of the 
 full value of $S_k$, while the form \eqref{eq:hard-boiled} says this constraint must be  sustained only by observing the ``boiled'' variables $Y_1, \ldots, Y_k, R$  (all of which take values in $[0,1]$). 
 %This is remarkable because 
% the set of all possible $S_k$ values can have arbitrarily large cardinality, including cardinalities that are much larger than the cardinality of all possible values of the variables $Y_1, \ldots, Y_k, R$. 
 In particular, all policies that satisfy \eqref{eq:system-constraints}-\eqref{eq:causal-measurable} are characterized according to the following choices: 
 \begin{enumerate} 
 \item Choose a single $\script{G}$-measurable random variable $R:\Omega\rightarrow [0,1]$. 
 \item Choose Borel measurable functions $\theta_k:\Omega_S\rightarrow [0,1]$ from which $Y_k=\theta_k(S_k)$ are defined for all $k \in \{1, 2, 3, \ldots\}$. 
 \item For each $k \in \{1, 2, 3, \ldots\}$, define a Borel measurable function $h_k:[0,1]^{k+1}\rightarrow \mathbb{R}^m$ such that 
 \begin{equation} \label{eq:over-boil} 
 h_k(\theta_1(S_1(\omega)), \theta_2(S_2(\omega)), \ldots, \theta_k(S_k(\omega)), R(\omega)) \in C(S_k(\omega)) \quad \forall \omega \in \Omega
 \end{equation} 
 \end{enumerate}  
 If the constraint \eqref{eq:over-boil} is impossible to meet, then no causal decision policy that meets the required measurability constraints exists. 
 Sufficient conditions for \eqref{eq:over-boil} are given in Section \ref{section:control} using two measurability assumptions that include the 
 existence of a \emph{measurable choice function}. Measurable choice is a classical problem in descriptive set theory and conditions for existence in certain cases are found in the  \emph{selection theorems} 
 of  \cite{blackwell-images}\cite{kuratawsi-selection}\cite{savage-gamble-selection}\cite{maitra-dp}\cite{aumann-measurable-choice}\cite{von-neumann-selection}\cite{srivastava-borel}. In particular, the works 
 \cite{savage-gamble-selection}\cite{maitra-dp}\cite{aumann-measurable-choice} use measurable choice to establish cost minimizing policies for economics and dynamic programming applications.  Our work gives a simple application to the multidimensional \emph{capacity region} in the opportunistic scheduling problem.  A theory of \emph{constructible sets} from \cite{constructible-book}, together with measurable choice, is used to refine the capacity results of \cite{sno-text}\cite{now}\cite{tass-server-allocation}. 
 We also apply classical pathological cases from descriptive set theory to show an example where a  nonmeasurable policy produces significantly larger time averages in comparison to any measurable policy.

\subsection{Related work}  

The Doob-Dynkin lemma is proven on page 603 in  \cite{doob-stochastic-processes} (see also 
Lemma 1.13 in \cite{kallenberg}, and \cite{williams-martingale}).    Recent discussion of this lemma is in \cite{doob-dynkin-similar-arxiv}. The Doob-Dynkin lemma 
can be used to directly characterize all  $\sigma(\script{H}_1\cup \script{H}_2)$-measurable functions $X:\Omega\rightarrow\mathbb{R}$ in the special case when $\script{H}_i=\sigma(Y_i)$ for some random variables $Y_i$ for $i \in \{1,2\}$. In that special case the Doob-Dynkin lemma implies $X=h(Y_1, Y_2)$.  The difficulty is that the sigma algebras $\script{H}_1$ and $\script{H}_2$ can be arbitrarily complex, including sigma algebras that cannot be generated by any real-valued random variable. An early version of this question was addressed by the author on  StackExchange in 
\cite{stackexchange-sigma-algebra} using Dynkin's multiplicative class 
theorem (see Theorem 18.51 in \cite{driver-monotone-class}) together with several techniques that are refined and generalized in the current paper.\footnote{The question of $X$ being 
$\sigma(\script{H}_1\cup \script{H}_2)$-measurable was posed by the author as a StackExchange question in \cite{stackexchange-sigma-algebra}.   Users initially conjectured the representation $X=h(Y_1, Y_2)$ was generally impossible but suggested proving a weaker representation by a monotone class argument; the strong result was eventually proven 
by the author using Dynkin's multiplicative class theorem \cite{driver-monotone-class}.}   Rather than using a multiplicative class argument, the current paper establishes a related sigma algebra fact that is of interest in its own right. 

%\footnote{Users on the mathoverflow site initially 
%conjectured that the representation $X=h(Y_1, Y_2)$ was generally impossible.  George Lowther on the site offered a proof of a related weaker result using a monotone class argument.  The original result was eventually proven true by the author using a monotone class argument.}  

For the probability space $(\Omega, \script{F}, P)$ used in the stochastic control 
problem, consider the special case when we are given some measurable space $(\Omega_Q, \script{F}_Q)$ and we are told $\script{G}=\sigma(Q)$ for some 
random element $Q:\Omega\rightarrow \Omega_Q$ that is measurable with respect to $(\Omega, \script{F})$ and $(\Omega_Q, \script{F}_Q)$.   The causal and measurable constraint \eqref{eq:causal-measurable} 
is thus equivalent to 
\begin{equation*}
X_k \mbox{ is $\sigma(\sigma(S_1)\cup \cdots \cup \sigma(S_k) \cup \sigma(Q))$-measurable} \quad \forall k \in \{1, 2, 3, \ldots\}
\end{equation*}
from which the Doob-Dynkin lemma immediately implies 
\begin{equation} \label{eq:intro}
X_k = h_k(S_1, S_2, \ldots, S_k, Q) \quad \forall k \in \{1, 2, 3, \ldots\}
\end{equation}  
for some measurable function $h_k:\Omega_S^k\times \Omega\rightarrow\mathbb{R}$.  However, the reason \eqref{eq:hard-boiled} is  stronger (and nontrivial) is that the $\script{G}$-measurable random variable $R:\Omega\rightarrow [0,1]$ takes values only in $[0,1]$ regardless of the complexity of the random element $Q$ that generates $\script{G}$; Similarly each $Y_j$ is $\script{H}_j$-measurable and takes values on $[0,1]$.

%The current paper deals with functions of potentially uncountably infinitely many variables using concepts of 
%product spaces. A recent work that explores versions of the Carath{\'e}odory extension theorem for a countable product of spaces
%is in \cite{general-countable-product-measure-arxiv} (see also  \cite{kakutani-product-space} for countable products of \emph{probability}  spaces). The work  \cite{general-countable-product-measure-arxiv}\cite{kakutani-product-space} restricts to countable products of spaces because it deals with questions of extending a product of measures to a single measure on the product space. In contrast, the current paper can handle uncountably infinite products because it 
%views a product of measurable spaces, together with its corresponding product sigma algebra, as yet another measurable space (for which no product measure is required).   In a probability theory context,  the probability triplet $(\Omega, \script{F}, P)$ is sufficient to measure events associated with a given measurable function $X:\Omega \rightarrow \Gamma$, where $(\Gamma, \script{G})$ is a given measurable space of interest (possibly defined by an uncountably infinite product of measurable spaces), in the sense that the function $\mu:\script{G}\rightarrow [0,1]$ defined by $\mu(G)=P[X^{-1}(G)]$ is itself a probability measure on $(\Gamma, \script{G})$. 

An important representation theorem related to \eqref{eq:hard-boiled} is given by Kallenberg in Proposition 5.13 of \cite{kallenberg}: There it is shown that if $X$ is a random element of a Borel space and $S$ is a random element of an arbitrary measurable space, and if the probability space is \emph{extended} (using standard product space concepts) to include a random variable $U$ that is uniformly distributed over $[0,1]$ and that is independent of everything else, then $X=g(S,W)$ almost surely, where $g$ is some measurable function and $W$ is some random variable that is uniformly distributed over $[0,1]$ and
independent of $S$.  It is not difficult to strengthen this result to \emph{surely} rather than \emph{almost surely}  
(this is done in Section \ref{section:kallenberg} for completeness).  When applied to the stochastic control problem, if we assume $(\Omega, \script{F}, P)$ is the already-extended space and $\script{G}=\sigma(U)$, the result immediately implies 
$$ X_k = g_k(S_k, W_k) \quad \forall k \in \{1, 2, 3, \ldots\}$$
where for each $k \in \{1, 2, 3, \ldots\}$, $g_k$ is a Borel measurable function and 
$W_k$ is a random variable that is uniformly distributed over $[0,1]$ and independent of $S_k$.   However, the $g_k$ functions cannot be viewed as defining a control policy because the value $W_k$ and its structure within the $g_k$ function can depend  on the realizations of $S_1, \ldots, S_{k-1}$. 

Selection theorems for measurable choice are developed by Blackwell and Ryll-Nardzewski \cite{blackwell-images}, Kuratowski and Ryll-Nardzewski \cite{kuratawsi-selection},  
and Von Neumann  \cite{von-neumann-selection} (see also \cite{srivastava-borel}\cite{banach-selection}). Measurable choice for economics and 
dynamic programming are considered by Maitra \cite{maitra-dp}, Aumann  \cite{aumann-measurable-choice}, and
Dubins and Savage \cite{savage-gamble-selection}.  For example, 
\cite{maitra-dp} considers a set $S$ for \emph{current states} and a set $A$ for \emph{action choices}, where $S$ is a Borel subset of a Polish space and $A$ is a compact metric space, and shows (see also \cite{savage-gamble-selection}) that if $u$ is a bounded upper semi-continuous function on $S \times A$ then there is a measurable choice function $\psi:S\rightarrow A$ such that 
$$ u(s, \psi(s)) = \max_{a \in A} u(s,a) \quad \forall s \in S$$
Continuous time control with measurable choice is in \cite{olech-selection-continuous-control}. 

Fundamental optimality 
properties for dynamic programming with general state and action sets 
are in \cite{blackwell-memoryless}\cite{blackwell-discount-dp}\cite{maitra-dp}\cite{schal-dp}. For example, Blackwell in  \cite{blackwell-memoryless} 
considers one step of a finite stage dynamic program with Borel spaces $A, S, H$ where $A$ is the set of possible actions, $S$ the set of current states, and $H$ the set of historical states from the past (see also \cite{savage-gamble-selection}). The one-step goal is to observe $s \in S$ and $h \in H$ and choose an action $a \in A$ to maximize a utility $u(s, a)$ (so the utility depends only on the current state and action). Mild conditions imply that for any policy that chooses $a$ as a measurable function of both $s$ and $h$, and for any $\epsilon>0$, 
there is a measurable \emph{memoryless strategy} that chooses  $a \in A$ based only on the current state $s$ that achieves utility at most $\epsilon$ worse (for almost all $a,h$ defined in a probabilistic sense). 
However, \cite{blackwell-memoryless} also gives a counter-example to show this is impossible without the mild conditions. This counter-example is similar in spirit to the example in Section \ref{section:nonmeasurable} of the current paper. However, the structure of our example is different: It treats the infinite horizon opportunistic scheduling problem;  It uses a different pathological set from descriptive set theory than the one used in \cite{blackwell-memoryless}; It compares a nonmeasurable policy to all possible measurable policies, rather than comparing a measurable policy of two variables to all possible measurable policies in one variable.  Optimality of stationary policies in multi-step dynamic programs over Borel spaces  is considered in \cite{blackwell-discount-dp}\cite{maitra-dp} and related nonstationary problems are 
in \cite{schal-dp}. Nonmeasurable gambling strategies  are treated in  \cite{savage-gamble-selection}.

Tassiulas and Ephremides establish the \emph{capacity region}  for a class of time-varying networks 
in \cite{tass-server-allocation} and prove that a max-weight rule stabilizes the network whenever possible. Capacity regions for more general systems that choose $X_k \in C(S_k)$ are treated in  \cite{now}\cite{sno-text}\cite{neely-power-network-jsac}, see also related problems of network utility maximization 
\cite{shroff-opportunistic}\cite{atilla-primal-dual-jsac}\cite{stolyar-greedy}\cite{neely-fairness-ton} 
and energy minimization \cite{neely-energy-it}. The general 
result in \cite{sno-text} makes implicit assumptions regarding measurability and probability space extension. 
The current paper refines a capacity theorem from \cite{sno-text} without 
extending the space by introducing two measurability assumptions, including a measurable choice assumption, together with a property of \emph{constructible sets} from \cite{constructible-book}. 

The field of descriptive set theory was initiated in the classic works of 
Souslin \cite{souslin-dst} and 
Lusin \cite{lusin-descriptive}. Souslin showed existence of a two dimensional 
Borel set that has a non-Borel projection onto the first dimension.  Examples of multidimensional 
Borel sets that do not contain a measurable choice function are developed by 
 Blackwell \cite{blackwell-borel-not-containing-graph},  Novikoff \cite{novikoff-borel-example}, Sierpi{\'n}ski  \cite{sierpinski-borel-example}, and  Addison \cite{addison-borel-example}
 (see also Example 5.1.7 in \cite{srivastava-borel}). In \cite{sierpinski-strange} Sierpi{\'n}ski constructs a subset of $[0,1]$ that has inner measure 0 and outer measure 1  
 (see also  \cite{alexander-strange}\cite{strange-uniform}). These classic pathological examples are used in Section \ref{section:counter-example} to show examples where nonmeasurable decisions can be used in the opportunistic scheduling problem to enable time averages that are superior to those achieved by any measurable policy.

\section{Preliminaries}

\subsection{Terminology} 

Let $\mathbb{N}=\{1, 2, 3, \ldots\}$ denote the natural numbers,  $\mathbb{R}$  the real numbers, and $\script{B}(\mathbb{R})$  the standard Borel sigma algebra on $\mathbb{R}$. For $A \in \script{B}(\mathbb{R})$ define $\script{B}(A) = \{B \in \script{B}(\mathbb{R}) : B \subseteq A\}$. 
%For $m \in \mathbb{N}$ and $A \in \script{B}(\mathbb{R})$ define 
%$$A^m = \{(a_1, \ldots, a_m) : a_i \in A \quad \forall i \in \{1, \ldots, n\}\}$$
%Let  $\script{B}(A^m)$ be the standard Borel sigma algebra on $A^m$. 
A \emph{measurable space} is a pair $(\Omega, \script{F})$ where $\Omega$ is a nonempty set and $\script{F}$ is a sigma algebra on $\Omega$.   Suppose $(\Omega_1, \script{F}_1)$ and $(\Omega_2, \script{F}_2)$ are two measurable spaces.  Let $\script{H}\subseteq \Omega_1$ be another sigma algebra on $\Omega_1$.  With respect to the measurable space 
$(\Omega_2, \script{F}_2)$, a function $g:\Omega_1\rightarrow \Omega_2$ is said to be \emph{$\script{H}$-measurable}
if 
$$ g^{-1}(A) \in \script{H} \quad \forall A \in \script{F}_2$$
where $g^{-1}(A) = \{\omega \in \Omega_1 : g(\omega) \in A\}$. 
%Lemma \ref{lem:collection} implies that a function $\phi:U\rightarrow V$ is $\script{H}$-measurable whenever  $\phi^{-1}(A) \in \script{H}$ for all sets $A$ in some collection of sets $\script{C}$ that satisfies $\sigma(\script{C})\supseteq \script{F}_V$. 
With respect to the two measurable 
spaces $(\Omega_1, \script{F}_1)$ and $(\Omega_2, \script{F}_2)$, a 
function $g:\Omega_1\rightarrow \Omega_2$ is said to be \emph{measurable} if it is $\script{F}_1$-measurable. 
%For simplicity, such a function will also be called \emph{measurable} when there is no ambiguity about the associated measurable spaces. 
Two measurable spaces $(\Omega_1, \script{F}_1)$ and $(\Omega_2, \script{F}_2)$ are \emph{isomorphic} if there is a bijective function $b:\Omega_1\rightarrow \Omega_2$ that is measurable and has a measurable inverse; such a function  is called an \emph{isomorphism}.  A measurable space $(\Omega, \script{F})$ is called a \emph{Borel space} if it is isomorphic to $(A, \script{B}(A))$ for some $A \in \script{B}([0,1])$.   
If $(\Omega_2, \script{F}_2)$ is a Borel space then a measurable function $g:\Omega_1\rightarrow\Omega_2$ is sometimes referred to as a \emph{Borel measurable function} as a reminder that the target space is a Borel space. 

%Lemma \ref{lem:collection} ensures that a function $\phi:U_1\rightarrow U_2$ is $\script{H}$-measurable whenever $\phi^{-1}(A) \in \script{H}$ for all sets $A \in \script{C}$, where $\script{C}$ is any collection of subsets of $U_2$ that satisfies $\script{G}_2 \subseteq \sigma(\script{C})$. 

Fix $J$ as a nonempty set (possibly uncountably infinite).  Let $(\Omega_j, \script{F}_j)$ be measurable spaces for each $j \in J$. Define 
$$ \times_{j \in J} \Omega_j = \{(x_j)_{j \in J} : x_j \in \Omega_j \quad \forall j \in J\}$$
Define $\script{C}$ as the collection of subsets of $\times_{j \in J} \Omega_j$ of the form $\times_{j \in J} A_j$ for some sets $A_j$ 
that satisfy: (i) $A_j \in \script{F}_j$ for all $j \in J$; (ii)  $A_j=\Omega_j$ for all but at most one index $j \in J$. Define the product sigma algebra on $\times_{j \in J} \Omega_j$, also called the cylindrical sigma algebra, as
$$ \otimes_{j \in J} \script{F}_j = \sigma(\script{C})$$
where $\sigma(\script{C})$ denotes the sigma algebra 
generated by the collection of sets $\script{C}$. 
For a given measurable space $(\Omega, \script{F})$ define $\Omega^J = \times_{j \in J} \Omega$ and define its product sigma algebra as   $\otimes_{j \in J} \script{F}$. 
A special case of interest is $[0,1]^J$ with product sigma algebra $\otimes_{j \in J} \script{B}([0,1])$ 
(this measurable space is a Borel space whenever $J$ is a finite or countably infinite set).

A \emph{probability space} is a triplet $(\Omega, \script{F}, P)$ where $(\Omega, \script{F})$ is a measurable space and $P:\script{F}\rightarrow[0,1]$ is a probability measure.  A \emph{random variable} is a measurable function $X:\Omega\rightarrow \mathbb{R}$. A \emph{random element} is a measurable function $S:\Omega\rightarrow \Omega_S$ where $(\Omega_S, \script{F}_S)$ is some given measurable space.  By $U \sim \script{U}[0,1]$ we mean that  $U:\Omega\rightarrow [0,1]$ is a random variable that is uniformly distributed over $[0,1]$.

\subsection{Standard results}

%\begin{lem} If $(V_1, \script{B}(V_1))$ and $(V_2, \script{B}(V_2))$ are Borel spaces and $b:V_1\rightarrow V_2$ is a measurable bijection, then $b$ is an isomorphism, that is, the inverse function $b^{-1}:V_2\rightarrow V_1$ is measurable. [See Theorem A1.7 in \cite{kallenberg} and Section I.3 in \cite{parthasarathy-borel}.]
%\end{lem}

\begin{lem} \label{lem:bijection} There is an isomorphism $\phi:[0,1]\rightarrow[0,1]^{\mathbb{N}}$.  
[See Theorem A.47 in \cite{breiman-probability} and Chapter 13 of \cite{dudley-probability}.] 
\end{lem} 

\begin{lem} \label{lem:uncountable-isomorphism} If $D$ is an uncountably infinite Borel measurable subset of a Borel space then there is an isomorphism $b:D\rightarrow [0,1]$. [This is a result of Kuratowski in \cite{kuratowski-borel},  see also statement and proof in  Theorem 3.3.13 of \cite{srivastava-borel}.]
\end{lem}

%\begin{lem} \label{lem:collection} Let $\Omega_1, \Omega_2$ be nonempty sets 
%and let $\script{F}_1$ be a sigma algebra on $\Omega_1$. 
%Let $\script{C}$ be any collection of subsets of $\Omega_2$. If $g:\Omega_1\rightarrow \Omega_2$ is a function then [See Prop. 3.2b in \cite{williams-martingale}]:
%$$ \left(g^{-1}(A) \in \script{F}_1 \quad \forall A \in \script{C}\right) \implies \left( g^{-1}(A) \in \script{F}_1 \quad \forall A \in \sigma(\script{C})\right)$$
%\end{lem}  

\begin{lem} \label{lem:standard2} Let $J$ be a nonempty set (possibly uncountably infinite).  Let $(\Omega, \script{F})$ and  
$(\Omega_j, \script{F}_j)$ for $j \in J$ be measurable spaces. 
Then [see similar Lemmas 1.7, 1.8 in \cite{kallenberg}]:

\begin{itemize} 
\item Composition: If $f:\Omega_1\rightarrow \Omega_2$ and $g:\Omega_2\rightarrow \Omega_3$ are measurable functions, the composition $g \circ f$ is measurable.  
%In particular, if $\script{H}$ is some other sigma algebra on $\Omega_1$ and if $f:\Omega_1\rightarrow \Omega_2$ is $\script{H}$-measurable, then $g \circ f$ is also $\script{H}$-measurable. 
 
%\item Multidimensional embedding: 
%Fix $L \subseteq J$ as a nonempty set. If $A \in \otimes_{j \in L}\script{F}_j$ then 
%$$ \{x \in \times_{j \in J} \Omega_j : (x_j)_{j\in L} \in A\} \in \otimes_{j \in J} \script{F}_j$$

\item Multidimensional expansion:  Let $Y_j:\Omega\rightarrow \Omega_j$ be measurable functions for each $j \in J$. The function $Y:\Omega\rightarrow  \times_{j \in J} \Omega_j$ given by $Y=(Y_j)_{j \in J}$ is measurable with respect to $(\Omega, \script{F})$ and $(\times_{j \in J} \Omega_j, \otimes_{j \in J} \script{F}_j)$.  
In particular, if $\script{H}_j$ is another sigma algebra on $\Omega$  for each $j \in J$,  and if $Y_j$ is $\script{H}_j$-measurable, then $Y$ is $\sigma(\cup_{j \in J} \script{H}_j)$-measurable.  
\end{itemize} 
\end{lem}

\section{Representation of Borel measurable functions} 

Throughout this section assume:  $(\Omega, \script{F})$ is a measurable space;   $J$ is a nonempty set (possibly uncountably infinite); 
$\script{H}_j \subseteq \script{F}$ is a  subsigma algebra on $\Omega$ for each $j \in J$.

\begin{prop} \label{thm:structure} Define $\script{C}$ as the set of functions $X:\Omega\rightarrow [0,1]$ of the form $X = h(\vec{Y})$
where $h:[0,1]^J\rightarrow [0,1]$ is measurable, $\vec{Y}=(Y_j)_{j\in J}$,  
and $Y_j:\Omega\rightarrow [0,1]$ is  $\script{H}_j$-measurable for each $j \in J$.  Define $\script{Z}$ as the following collection of subsets of  $\Omega$: 
$$ \script{Z} = \{X^{-1}(B) \subseteq \Omega : B \in \script{B}([0,1]), X \in \script{C}\}$$
Then 

a) $\script{Z}$ is a sigma algebra on $\Omega$.  

b) $\sigma(\cup_{j\in J}\script{H}_j)= \script{Z}$. 

c)  $X:\Omega\rightarrow [0,1]$ is $\sigma(\cup_{j\in J}\script{H}_j)$-measurable if and only if $X \in \script{C}$. 

\end{prop} 

\begin{proof} (Part (a) of Proposition \ref{thm:structure}) 
We show  $\script{Z}$ satisfies the three properties of a sigma algebra on $\Omega$: 

\begin{enumerate} 
\item To show $\Omega \in \script{Z}$, define
the measurable functions $h=0$,  $Y_j=0$ for all $j \in J$, and  
$X = h(\vec{Y})=0 \in \script{C}$.  Define $B = [0,1] \in \script{B}([0,1])$. Then
$\Omega =  X^{-1}(B) \in \script{Z}$.

\item Fix $A \in \script{Z}$. We want to show $A^c \in \script{Z}$.  Since $A \in \script{Z}$ there exists $X\in \script{C}$ and $B \in \script{B}([0,1])$ such that $A = X^{-1}(B)$. Then 
$A^c = X^{-1}(B^c)  \in \script{Z}$. 

\item Let $\{A_n\}_{n=1}^{\infty}$ be an infinite sequence of sets in $\script{Z}$. We want to show $\cup_{n=1}^{\infty} A_n \in \script{Z}$.  
It suffices to show $\cap_{n=1}^{\infty} A_n^c \in \script{Z}$. For each positive integer $n$ there exists $X_n\in \script{C}$ and $B_n\in \script{B}([0,1])$ such that 
$A_n = X_n^{-1}(B_n)$ 
and so 
$A_n^c = X_n^{-1}(B_n^c)$. 
Let $\phi:[0,1]\rightarrow[0,1]^{\mathbb{N}}$ be an isomorphism (recall Lemma \ref{lem:bijection}). Define 
\begin{align}
X&=\phi^{-1}\left((X_n)_{n=1}^{\infty}\right) \label{eq:X-proof1} \\
B &= \phi^{-1}\left(\times_{n=1}^{\infty} B_n^c  \right) \nonumber
\end{align}
Since $\phi^{-1}$ maps measurable sets to measurable sets
we have $B \in  \script{B}([0,1])$. Then: 
\begin{align*}
\cap_{n=1}^{\infty} A_n^c &=\left\{\omega \in \Omega: X_n(\omega) \in B_n^c \quad \forall n \in \mathbb{N}\right\} \\
&= \left\{\omega \in \Omega : \phi^{-1}\left((X_n(\omega))_{n=1}^{\infty}\right) \in \phi^{-1}\left(\times_{n=1}^{\infty} B_n^c  \right)\right\}\\
&= X^{-1}(B)
\end{align*}
Considering the structure of set $\script{Z}$, 
it remains to show that $X \in \script{C}$.  Fix $n \in \mathbb{N}$. Since $X_n \in \script{C}$ we have 
\begin{equation} \label{eq:Wn-structure}
 X_n = h^{(n)}(\vec{Y}^{(n)}) 
 \end{equation} 
for some measurable function $h^{(n)}:[0,1]^J\rightarrow [0,1]$ and some 
$\vec{Y}^{(n)}=(Y^{(n)}_{j})_{j \in J}$
such that $Y^{(n)}_j:\Omega\rightarrow [0,1]$ is $\script{H}_j$-measurable for all $j \in J$. 
For each $j \in J$ define $W_j:\Omega\rightarrow [0,1]$ by 
\begin{equation} \label{eq:W-proof1}
W_j = \phi^{-1}(Y_{j}^{(1)}, Y_{j}^{(2)}, Y_{j}^{(3)}, \ldots)
\end{equation} 
Note that $W_j$ is a composition of the measurable function $\phi^{-1}:[0,1]^{\mathbb{N}}\rightarrow [0,1]$
with the $\script{H}_j$-measurable function $Z:\Omega\rightarrow [0,1]^{\mathbb{N}}$ given 
by $Z(\omega) = (Y_{j}^{(1)}(\omega), Y_{j}^{(2)}(\omega), Y_j^{(3)}(\omega), \ldots)$ and hence $W_j$ is  
itself $\script{H}_j$-measurable (recall Lemma \ref{lem:standard2}).  
Write function $\phi$ according to its components
 $\phi=(\phi_1, \phi_2, \phi_3, \ldots)$ 
 and note that each component function $\phi_n:[0,1]\rightarrow[0,1]$ is measurable.   For each $j \in J$ we have from \eqref{eq:W-proof1}
\begin{align*}
(Y_j^{(1)}, Y_j^{(2)}, Y_j^{(3)}, \ldots) &=\phi(W_j) \\
&=(\phi_1(W_j), \phi_2(W_j), \phi_3(W_j), \ldots) 
\end{align*}
and so $Y_{j}^{(n)} = \phi_n(W_j)$ for all $j \in J, n \in \mathbb{N}$, that is, 
$$\vec{Y}^{(n)} = \left(\phi_n(W_j)\right)_{j \in J} $$
Substituting the above equality into  \eqref{eq:Wn-structure} yields
\begin{equation}\label{eq:andso} 
X_n  = h^{(n)}((\phi_n(W_j))_{j \in J})
\end{equation} 
Define the function $\alpha^{(n)}:[0,1]^J\rightarrow [0,1]$ for each $x=(x_j)_{j \in J}$ by 
$$ \alpha^{(n)}(x) = h^{(n)}((\phi_n(x_j))_{j \in J})$$
Define $\vec{W}=(W_j)_{j \in J}$.  Using this and the definition of $\alpha^{(n)}$  in \eqref{eq:andso} gives: 
\begin{equation} \label{eq:Wn}
X_n = \alpha^{(n)}(\vec{W})
\end{equation} 
Define the function  $h:[0,1]^J\rightarrow [0,1]$ by 
$$ h(x) = \phi^{-1}(\alpha^{(1)}(x), \alpha^{(2)}(x), \alpha^{(3)}(x), \ldots) \quad \forall x \in [0,1]^J$$
The functions $\alpha^{(n)}$ and $h$ are formed by compositions and multidimensional expansions of 
measurable functions and so they are themselves measurable (recall Lemma \ref{lem:standard2}).  By definition of $h$ it holds that 
\begin{align*}
h(\vec{W}) &= \phi^{-1}(\alpha^{(1)}(\vec{W}), \alpha^{(2)}(\vec{W}), \alpha^{(3)}(\vec{W}), \ldots)\\
&\overset{(a)}{=} \phi^{-1}(X_1, X_2, X_3, \ldots)\\
&\overset{(b)}{=} X
\end{align*}
where (a) holds by substituting \eqref{eq:Wn}; (b) holds by definition of $X$ in \eqref{eq:X-proof1}. Thus, $X \in \script{C}$. 
\end{enumerate} 
\end{proof} 

\begin{proof} (Part (b) of Proposition \ref{thm:structure}) To show that $\script{Z}\subseteq \sigma(\cup_{j \in J} \script{H}_j)$, fix $A \in \script{Z}$.  By definition of $\script{Z}$, there exists $B \in \script{B}([0,1])$ and $X \in \script{C}$ such that $A = X^{-1}(B)$,  
where $X=h(\vec{Y})$ for some 
measurable function $h:[0,1]^J\rightarrow [0,1]$ and some 
vector-valued function $\vec{Y}=(Y_j)_{j\in J}$ composed of $\script{H}_j$-measurable functions $Y_j:\Omega\rightarrow[0,1]$ for each $j\in J$. Thus
\begin{align} 
A &= X^{-1}(B) \nonumber\\
 &= \{\omega \in \Omega: h(\vec{Y}) \in B\}\label{eq:A-here}
\end{align} 
Lemma \ref{lem:standard2} ensures that $h(\vec{Y})$ is  $\sigma(\cup_{j \in J} \script{H}_j)$-measurable, and so the right-hand-side of \eqref{eq:A-here} is a set in $\sigma(\cup_{j \in J} \script{H}_j)$, which implies the desired conclusion $A \in \sigma(\cup_{j \in J} \script{H}_j)$. 

We now show $\sigma(\cup_{j \in J} \script{H}_j) \subseteq \script{Z}$. 
Fix $m \in J$. Let $A_m$ be a subset of $\Omega$ 
such that $A_{m}\in \script{H}_{m}$.  Define $\vec{Y}=(Y_j)_{j\in J}$ by $Y_j=0$ if $j\neq m$ and  
$$Y_{m}(\omega)= \left\{\begin{array}{cc}
1 & \mbox{ if $\omega \in A_{m}$} \\
0 & \mbox{ else} 
\end{array}\right.$$
It is clear that $Y_j$ is $\script{H}_j$-measurable for all $j \in J$. Define the measurable function $h:[0,1]^J\rightarrow [0,1]$ by $h((x_j)_{j\in J}) = x_{m}$. Define $B = \{1\} \in \script{B}([0,1])$.  Then 
$$ A_m = \{\omega \in \Omega : h(\vec{Y}(\omega))  \in B\}$$
so by definition of $\script{Z}$ we have $A_{m} \in \script{Z}$.  This holds for all $m \in J$ and $A_{m} \in \script{H}_{m}$ so 
$$ \cup_{j \in J} \script{H}_j \subseteq \script{Z}$$
Taking the sigma algebra of both sides gives
$$ \sigma(\cup_{j \in J} \script{H}_j) \subseteq \sigma(\script{Z})$$
Part (a) implies that $\sigma(\script{Z})=\script{Z}$, which completes the proof. 
\end{proof}

\begin{proof} (Part (c) of Proposition \ref{thm:structure}) Suppose $X \in \script{C}$. Then $X=h(\vec{Y})$ for
some measurable $h$ and for $\vec{Y}=(Y_j)_{j\in J}$ with $Y_j:\Omega\rightarrow [0,1]$ being $\script{H}_j$-measurable for all $j \in J$.  Lemma \ref{lem:standard2} implies that $X$ is $\sigma(\cup_{j \in J} \script{H}_j)$-measurable.

Now suppose $X:\Omega\rightarrow [0,1]$ is $\sigma(\cup_{j \in J} \script{H}_j)$-measurable. It is well known that $X$ is the pointwise limit of \emph{simple} functions $X_m$, so that  
\begin{equation}\label{eq:limitXm}
X(\omega)  = \lim_{m\rightarrow\infty} X_m(\omega) \quad \forall \omega \in \Omega
\end{equation} 
where for each positive integer $m$ the function $X_m:\Omega\rightarrow [0,1]$ has the form 
\begin{equation} \label{eq:Xm}
X_m = \sum_{i=1}^{k_m} v_{i,m}1_{\{X \in I_{i,m}\}}
\end{equation} 
where $k_m$ is some positive integer; $I_{1,m}, I_{2,m}, \ldots, I_{k_m,m}$  
are  some disjoint sets in $\script{B}([0,1])$ whose union is $[0,1]$; $1_A$ is an indicator function that is 1 if event $A$ is true and $0$ else; $v_{i,m}$ are some real numbers in $[0,1]$ for each $i \in \{1, \ldots, k_m\}$. 

Since $X$ is $\sigma(\cup_{j \in J} \script{H}_j)$-measurable, we have for each positive integer $m$ and each $i\in \{1, \ldots, k_m\}$: 
$$\{X \in I_{i,m}\} \in \sigma(\cup_{j \in J} \script{H}_j)= \script{Z}$$
where the final equality holds by part (b). It follows by definition of $\script{Z}$ that 
\begin{equation}\label{eq:equality}
\{X \in I_{i,m}\} = \{\omega \in \Omega: X_{i,m}(\omega) \in B_{i,m}\}
\end{equation} 
for some $B_{i,m} \in \script{B}([0,1])$ and some $X_{i,m} \in \script{C}$. Substituting \eqref{eq:equality} into  \eqref{eq:Xm} 
and using \eqref{eq:limitXm} we obtain 
\begin{equation}\label{eq:pluglimit}
 X = \limsup_{m\rightarrow\infty} \sum_{i=1}^{k_m} v_{i,m} 1_{\{X_{i,m} \in B_{i,m}\}}
 \end{equation} 
where we have used the fact that the 
limit exists and so must be equal to the $\limsup$. 

By definition of $\script{C}$, each function $X_{i,m} \in \script{C}$ has the form
\begin{equation} \label{eq:Xim}
X_{i,m} = h^{(i,m)}((Y_j^{(i,m)})_{j\in J})
\end{equation} 
for measurable functions $h^{(i,m)}:[0,1]^J\rightarrow [0,1]$ and some $\script{H}_j$-measurable functions $Y_j^{(i,m)}:\Omega\rightarrow [0,1]$ for $j \in J$. Let $L$ be the (countably infinite) set of all indices $(i,m)$ such that $m\in \mathbb{N}$ and $i \in \{1, \ldots, k_m\}$. Let $\phi:[0,1]\rightarrow[0,1]^L$ be an isomorphism. 
For each $j\in J$ define $W_j = \phi^{-1}((Y_j^{(i,m)})_{(i,m)\in L})$. Since $W_j$ is the composition of the measurable function $\phi^{-1}$ with the multidimensional expansion of $\script{H}_j$-measurable functions, it is itself $\script{H}_j$-measurable (recall Lemma \ref{lem:standard2}).   Define $\phi_{i,m}$ as the $(i,m)$ component of the $\phi$ function for each $(i,m) \in L$. Then from \eqref{eq:Xim}
\begin{align*}
X_{i,m} &= h^{(i,m)}\left((\phi_{i,m}(W_j))_{j\in J}\right)
\end{align*}
 Define $\vec{W}=(W_j)_{j\in J}$. Then  
\begin{equation}\label{eq:plugchug}
(X_{i,m})_{(i,m)\in L} = \alpha(\vec{W})
\end{equation} 
where $\alpha:[0,1]^J\rightarrow [0,1]^L$ is the measurable function defined for $x=(x_j)_{j\in J}$ by component functions $\alpha_{i,m}(x)$ for each $(i,m)\in L$ by 
$$ \alpha_{i,m}(x)= h^{(i,m)}\left((\phi_{i,m}(x_j))_{j\in J}\right)$$

Define the measurable function $g:[0,1]^L\rightarrow [0,1]$ for each $x=(x_{i,m})_{(i,m)\in L}$ by 
$$g(x) = \limsup_{m\rightarrow\infty} \sum_{i=1}^{k_m}v_{i,m}1_{\{x_{i,m} \in B_{i,m}\}}$$
where we observe the limsup is in the set $[0,1]$ because for each $m$, all $v_{i,m}$ values are in  $[0,1]$ and at most one term in the sum is nonzero. 
It follows that 
\begin{align*}
 X &\overset{(a)}{=} g((X_{i,m})_{(i,m)\in L})\\
 &\overset{(b)}{=} g(\alpha(\vec{W}))
 \end{align*}
 where (a) holds by \eqref{eq:pluglimit}; (b) holds by \eqref{eq:plugchug}.  We can now define the measurable function $h:[0,1]^J\rightarrow [0,1]$ by $h(x) = g(\alpha(x))$ and we see that $X=h(\vec{W})$, where $\vec{W}=(W_j)_{j\in J}$ for $W_j:\Omega\rightarrow [0,1]$ being $\script{H}_j$-measurable for all $j \in J$.  It follows that $X$ has the required form for inclusion in the set $\script{C}$. 
\end{proof}

Now fix $K$ as a finite or countably infinite set.  For each $k \in K$ let $(V_k, \script{F}_k)$ be a Borel space. We consider measurable functions $X_k:\Omega \rightarrow V_k$. 

\begin{prop} \label{cor:main}  Fix $J$ as a nonempty set (possibly uncountably infinite). For each $j \in J$, let $\script{H}_j\subseteq \script{F}$ be a sigma algebra on $\Omega$.  Fix functions $X_k:\Omega\rightarrow V_k$ for $k \in K$, where $(V_k, \script{F}_k)$ are given Borel spaces.   For each $k \in K$, fix $J_k\subseteq J$. Then 
\begin{equation} \label{eq:constraints1} 
X_k \mbox { is $\sigma(\cup_{j \in J_k} \script{H}_j)$-measurable} \quad \forall k \in K
\end{equation} 
if and only if for each $k \in K$ we have 
\begin{equation} \label{eq:form}
X_k=h_k((Y_j)_{j \in \tilde{J}_k})
\end{equation} 
where  $h_k:[0,1]\rightarrow V_k$ is some measurable function,    $Y_j:\Omega\rightarrow [0,1]$ are some $\script{H}_j$-measurable functions for each $j \in J$, and $\tilde{J}_k$ is a finite or countably infinite subset of $J_k$ for each $k \in K$.
\end{prop} 

\begin{proof} 
For the reverse direction, it is clear from Lemma \ref{lem:standard2} 
that if $(X_k)_{k\in K}$ has the given 
form $X_k=h_k((Y_j)_{j \in \tilde{J}_k})$   then \eqref{eq:constraints1} holds.
To prove the forward direction, suppose that \eqref{eq:constraints1} holds. Fix $k \in K$. 
Since $(V_k, \script{F}_k)$ is a Borel space, 
there is a set $D_k\in \script{B}([0,1])$ and an isomorphism $b_k:V_k\rightarrow D_k$. 
Define $Z_k:\Omega\rightarrow [0,1]$ by $Z_k=b_k(X_k)$. Lemma \ref{lem:standard2} implies that $Z_k$ is $\sigma(\cup_{j \in J_k}\script{H}_j)$-measurable. By Proposition \ref{thm:structure} we have 
$Z_k = g^{(k)}\left((Y_j^{(k)})_{j \in J_k}\right)$
with $Y_j^{(k)}:\Omega\rightarrow [0,1]$ being $\script{H}_j$-measurable for all $j \in J$, and $g^{(k)}:[0,1]^{J_k}\rightarrow [0,1]$ is measurable.  For every such real-valued measurable function $g^{(k)}$, there is a countable subset $\tilde{J}_k\subseteq J_k$ for which the function only depends on the variables $y_j$ for $j \in \tilde{J}_k$ [see, for example,  related Exercise 1.1.22 in \cite{dembo-notes} and Section 3.13d in \cite{williams-martingale}].  Thus, we modify the $g^{(k)}$ functions to $f^{(k)}:[0,1]^{\tilde{J}_k}\rightarrow [0,1]$ measurable for which 
\begin{equation} \label{eq:Z-k-here}
Z_k = f^{(k)}\left((Y_j^{(k)})_{j \in \tilde{J}_k}\right)
\end{equation} 
Let $\phi:[0,1]\rightarrow[0,1]^K$ be an isomorphism. For each $j \in J$ define 
$$ Y_j = \phi^{-1}((Y_j^{(k)})_{k\in K})$$
Since each function $Y_j^{(k)}$ is $\script{H}_j$-measurable, $Y_j$ is also $\script{H}_j$-measurable (recall Lemma \ref{lem:standard2}).   For each $k \in K$ let $\phi_k$ denote the $k$th component of $\phi$.  Then 
$$ \phi_k(Y_j) = Y_j^{(k)}$$
which gives by substitution into \eqref{eq:Z-k-here}: 
\begin{align}
 Z_k &= f^{(k)}\left( (\phi_k(Y_j))_{j \in \tilde{J}_k} \right) \nonumber \\
 &=\alpha^{(k)}\left((Y_j)_{j \in \tilde{J}_k}\right) \label{eq:Z-k-here2} 
 \end{align}
 where  $\alpha^{(k)}:[0,1]^{\tilde{J}_k}\rightarrow [0,1]$ is defined as the measurable function for each $x=(x_j)_{j\in \tilde{J}_k}$ by 
 $$ \alpha^{(k)}(x) = f^{(k)}\left((\phi_k(x_j))_{j\in \tilde{J}_k}\right)$$
Substituting the definition $Z_k=b_k(X_k)$ into the left-hand-side of \eqref{eq:Z-k-here2} gives 
$$b_k(X_k) = \alpha^{(k)}\left((Y_j)_{j \in \tilde{J}_k}\right)$$
Taking $b_k^{-1}(\cdot)$ of both sides  gives 
$$ X_k = b_k^{-1}\left(\alpha^{(k)}\left((Y_j)_{j \in \tilde{J}_k}\right)\right) $$
This holds for all $k \in K$ and   has  the desired form $X_k=h_k\left((Y_j)_{j \in \tilde{J}_k}\right)$ when the measurable function  $h_k:[0,1]^{\tilde{J}_k}\rightarrow V_k$ is defined  by $h_k(x) = b_k^{-1}(\alpha^{(k)}(x))$ for all $x \in [0,1]^{\tilde{J}_k}$.
\end{proof} 

\begin{cor}  \label{cor:function} Let $(V, \script{F}_V)$ be a Borel space. Let $J$ be a nonempty set 
and let $(\Omega_j, \script{F}_j)$ be measurable spaces for each $j \in J$. If  $f:\times_{j \in J} \Omega_j\rightarrow V$ is a measurable function with respect to  $(\times_{j \in J} \Omega_j, \otimes_{j \in J} \script{F}_j)$ and $(V, \script{F}_V)$ and $\omega = (\omega_j)_{j \in J}$ then 
$$ f(\omega) = h((\theta_j(\omega_j))_{j \in \tilde{J}}) \quad \forall \omega \in \times_{j \in J} \Omega_j$$
where $\tilde{J}\subseteq J$ is a finite or countably infinite set, $\theta_j:\Omega_j\rightarrow [0,1]$ is a measurable function for each $j \in \tilde{J}$, and 
$h:\times_{j \in \tilde{J}}\Omega_j \rightarrow V$ is some measurable function with respect to  $(\times_{j \in \tilde{J}} \Omega_j, \otimes_{j \in J} \script{F}_j)$ and $(V, \script{F}_V)$.
\end{cor} 

\begin{proof} Define $S_j:\Omega\rightarrow \Omega_j$ by $S_j(\omega) = \omega_j$ for $j \in J$. Define $\script{H}_j=\sigma(S_j)$.  Then $f$ is $\sigma(\cup_{j\in J} \script{H}_j)$-measurable and Proposition \ref{cor:main} implies $f=h((Y_j)_{j \in \tilde{J}})$ for a countable subset $\tilde{J}\subseteq J$,  a measurable function $h$,  and for $Y_j$ being $\sigma(S_j)$ measurable for each $j \in \tilde{J}$.  The Doob-Dynkin lemma implies $Y_j=\theta_j(S_j)=\theta_j(\omega_j)$ for $j \in \tilde{J}$. 
\end{proof}

\section{Stochastic control} \label{section:control}

Throughout this section we fix a probability triplet $(\Omega, \script{F}, P)$. Let 
$(\Omega_S, \script{F}_S)$ be a measurable space and let  $(\Omega_X, \script{F}_X)$ be a Borel space.  Consider a discrete time system that evolves over time slots $k \in \{1, 2, 3, \ldots\}$.  Let $(S_k)_{k=1}^{\infty}$ be a sequence of random elements of the form $S_k:\Omega\rightarrow \Omega_S$. The value $S_k$  represents a system characteristic or state at time $k$. 
Let $\script{G}\subseteq \script{F}$ be a sigma algebra on $\Omega$ that is used as a source of randomness to facilitate stochastic decisions. Let $(X_k)_{k=1}^{\infty}$ be a sequence of random elements of the form $X_k:\Omega\rightarrow \Omega_X$. Each $X_k$ represents a decision that is made at time $k$ based on observing $S_1, \ldots, S_k$.   Assume decisions for each step $k$ are made to ensure  
\begin{align}
&X_k \mbox{ is $\sigma(\sigma(S_1)\cup \cdots \cup \sigma(S_k) \cup \script{G})$-measurable} \label{eq:control1} \\
&X_k \in C(S_k) \quad  \label{eq:control2}
\end{align} 
where $C:\Omega_S\rightarrow Pow(\Omega_X)$ is a set-valued map  and $Pow(\Omega_X)$ is the set of all subsets of $\Omega_X$. Constraint \eqref{eq:control1} is the \emph{causal and measurable} constraint. Constraint \eqref{eq:control2} is a system constraint that restricts the $X_k$ value to a set that depends on  $S_k$. Values of $S_{k+1}$ are determined by some probability rule on the system and are possibly dependent on $S_1, \ldots, S_k$ and $X_1, \ldots, X_k$.  A special case is when $S_k$ represents the state of a discrete time Markov chain and there is some transition probability kernel that specifies the conditional distribution of $S_{k+1}$ given $S_k$ and $X_k$. 

Decisions $X_k$ can be vector valued with components that represent power expenditures, costs, or rewards incurred or earned by different parts of the system at time $k$.   We want to characterize all decision elements $(X_k)_{k=1}^{\infty}$ that satisfy \eqref{eq:control1}-\eqref{eq:control2}. Proposition \ref{cor:main} ensures that if \eqref{eq:control1}-\eqref{eq:control2} hold then  
$$X_k = h_k(Y_1, \ldots, Y_k, R) \in C(S_k)  \quad \forall k \in \mathbb{N}$$
for  some Borel measurable functions $h_k:\Omega_S^k\times [0,1]\rightarrow \Omega_X$, some $\script{G}$-measurable random variable $R:\Omega\rightarrow [0,1]$, 
and some  random variables $Y_k =\theta_k(S_k)$ for some measurable functions $\theta_k:\Omega_S\rightarrow [0,1]$. It immediately follows that 
\begin{equation} \label{eq:Xg} 
X_k = g_k(S_1, \ldots, S_k, R) \in C(S_k) \quad \forall k \in \mathbb{N}
\end{equation} 
where $g_k:\Omega_S^k\times [0,1]\rightarrow \Omega_X$ is defined 
$$g_k(s_1, \ldots, s_k, r) = h_k(\theta_1(s_1), \ldots, \theta_k(s_k), r)$$
Consider the following additional assumptions: 

\begin{assumption} \label{assumption:1} There is a deterministic measurable choice function $\psi:\Omega_S\rightarrow \Omega_X$ such that $\psi(s) \in C(s)$ for all $s \in \Omega_S$.
\end{assumption} 

\begin{assumption}  \label{assumption:2}  
$\{(s,x)  \in \Omega_S \times \Omega_X: x \in C(s)\} \in \script{F}_S\otimes \script{F}_X$
\end{assumption} 

Both assumptions hold if $\Omega_S$ is a finite or countably infinite set, $\script{F}_S=Pow(\Omega_S)$, and  $C(s)$ is a nonempty subset of $\script{F}_X$ for each $s \in \Omega_S$.  Assumptions \ref{assumption:1}-\ref{assumption:2} also hold in the case when a vector of resources
$P_k \in \mathbb{R}^a$ (such as power allocations) is chosen on each slot $k\in \mathbb{N}$ 
and affects a vector of rewards $R_k\in \mathbb{R}^b$ (such as transmission rates over links of a communication system) via $R_k=f(S_k,P_k)$, where $a, b$ are given positive integers, 
$\Omega_P$ is a given Borel measurable subset of $\mathbb{R}^a$,   $f:\Omega_S\times \Omega_P\rightarrow \mathbb{R}^b$ is a given measurable function, and 
\begin{equation} \label{eq:example-12} 
C(s) = \{(p, f(s,p)) \in \mathbb{R}^{a+b} : p \in \Omega_P\} \quad \forall s \in \Omega_S
\end{equation} 
Indeed, Assumption \ref{assumption:1} holds for  \eqref{eq:example-12} because 
$\psi(s)=(0, f(s,0))$ is a deterministic measurable choice function;  
Assumption \ref{assumption:2} can be seen to hold for  \eqref{eq:example-12} by defining 
the measurable function  
$g:\Omega_S\times \Omega_P \times \mathbb{R}^b\rightarrow \mathbb{R}^b$ by $g(s,p,r) = r-f(s,p)$
and observing that $g^{-1}(\{0\})$ is measurable. 
More
general sufficient conditions for existence of a deterministic measurable choice function are given in the \emph{selection theorems} of Blackwell and Ryll-Nardzewski \cite{blackwell-images},  Kuratowski and Ryll-Nardzewski \cite{kuratawsi-selection},  Dubins and Savage \cite{savage-gamble-selection}, Maitra \cite{maitra-dp}, Aumann  \cite{aumann-measurable-choice}, Sch{\"a}l \cite{schal-selection}, Von Neumann \cite{von-neumann-selection},  Srivastava \cite{srivastava-borel}, and Cascales, Kadets, Rodr{\'i}guez \cite{banach-selection}.  

\begin{lem} Suppose Assumptions \ref{assumption:1} and \ref{assumption:2} hold.  The sequence $(X_k)_{k=1}^{\infty}$ of Borel measurable random elements of the form $X_k:\Omega\rightarrow \Omega_X$ satisfies \eqref{eq:control1}-\eqref{eq:control2} if and only if  there are  measurable functions  
$v_k:\Omega_S^k\times [0,1]\rightarrow \Omega_X$ for each $k \in \mathbb{N}$ such that 
\begin{equation} \label{eq:desired-v}
 v_k(s_1, \ldots, s_k, r) \in C(s_k) \quad  \forall (s_1, \ldots, s_k, r) \in \Omega_S^k\times [0,1]
 \end{equation} 
and a $\script{G}$-measurable random variable $R:\Omega\rightarrow [0,1]$ 
such that 
\begin{equation} \label{eq:Xv} 
X_k=v_k(S_1, \ldots, S_k, R) \quad \forall k \in \mathbb{N}
\end{equation} 
\end{lem} 

\begin{proof} Suppose $(X_k)_{k=1}^{\infty}$ satisfy \eqref{eq:control1}-\eqref{eq:control2}. 
Then \eqref{eq:Xg} holds for some measurable functions
$g_k:\Omega_S^k\times [0,1]\rightarrow \Omega_X$ and some $\script{G}$-measurable random variable $R:\Omega\rightarrow [0,1]$. Define $v_k:\Omega_S^k\times [0,1]\rightarrow \Omega_X$ by 
$$v_k(s_1, \ldots, s_k, r) = \left\{\begin{array}{cc}
g_k(s_1, \ldots, s_k, r) & \mbox{ if $g_k(s_1, \ldots, s_k, r) \in C(s_k)$} \\
\psi(s_k) & \mbox{ else} 
\end{array}\right.$$
Assumptions \ref{assumption:1},  \ref{assumption:2},  and measurability of $g_k$ imply that $v_k$ is 
measurable.  Since $\psi(s)\in C(s)$ for all $s \in \Omega_S$, function $v_k$ satisfies \eqref{eq:desired-v}. By \eqref{eq:Xg} and definition of $v_k$ we obtain \eqref{eq:Xv}.  

Conversely, suppose there are $v_k$ functions and a random variable $R$ that satisfy 
\eqref{eq:desired-v}-\eqref{eq:Xv}. Properties \eqref{eq:desired-v}-\eqref{eq:Xv} imply $X_k\in C(S_k)$ for all $k$, while measurability of $v_k$ and the structure 
$X_k=v_k(S_1, \ldots, S_k, R)$ ensure (by the Doob-Dynkin lemma) 
that $X_k$ is $\sigma(S_1, \ldots, S_k, R)$-measurable.  Since $\sigma(R)\subseteq \script{G}$ it holds that $X_k$ is $\sigma(\sigma(S_1)\cup\cdots\cup\sigma(S_k)\cup \script{G})$-measurable, so that \eqref{eq:control1}-\eqref{eq:control2} hold. 
\end{proof} 

The $v_k$ functions and the random variable $R:\Omega\rightarrow[0,1]$ in the above 
result completely specify a causal and measurable control policy: At time $0$, generate a $\script{G}$-measurable random variable $R:\Omega\rightarrow [0,1]$. At each step $k \in \{1, 2, 3, \ldots\}$, observe $(S_1, \ldots, S_k)$ and make the decision $X_k=v_k(S_1, \ldots, S_k, R)$. The above lemma ensures that, if Assumptions \ref{assumption:1}-\ref{assumption:2} hold, all  policies that satisfy \eqref{eq:control1}-\eqref{eq:control2} can be specified in this way.

\subsection{Another representation} \label{section:kallenberg}

The following representation theorem from Kallenberg \cite{kallenberg} bears  some resemblance to \eqref{eq:Xg} and uses the concept of a \emph{randomization variable $U$}.%\footnote{The statement of Prop. 5.13 in \cite{kallenberg} does not explicitly mention the randomization variable $U$. However, the need for $U$ is implicit in the discussion on extending the state space that is 
%given before that proposition. 
%The $\sigma(X,S,U)$-measurability of $R$ is not explicitly mentioned in Prop. 5.13 of \cite{kallenberg} but is implicit in its proof which constructs $R$ in stages from measurable functions of $X,S,U$.} 

\begin{thm} \label{thm:kallenberg} (Proposition 5.13 in  \cite{kallenberg}) Fix $(\Omega, \script{F}, P)$ as a probability triplet and let $X:\Omega\rightarrow \Omega_X$ and $S:\Omega\rightarrow \Omega_S$ be random elements where $(\Omega_X, \script{F}_X)$ is a 
Borel space and $(\Omega_S, \script{F}_S)$ is a measurable space. Suppose there is a random variable $U\sim \script{U}[0,1]$ that is independent of $(S,X)$ ($U$ is called a \emph{randomization variable}). Then  
$$ X=f(S,R) \quad \mbox{almost surely} $$ 
for some measurable function $f:\Omega_S\times [0,1]\rightarrow \Omega_X$ and some 
random variable $R \sim \script{U}[0,1]$ that is independent of $S$. Further, 
$R$ is   $\sigma(S,X,U)$-measurable. 
\end{thm}

The next simple corollary changes ``almost surely'' to ``surely.''  

\begin{cor} \label{cor:sure-kallenberg} Under the same assumptions as Theorem \ref{thm:kallenberg} we can ensure $X=g(S,W)$ surely for some measurable function $g:\Omega_S\times [0,1]\rightarrow \Omega_X$ and some 
random variable $W \sim \script{U}[0,1]$ that is independent of $S$ and that is 
$\sigma(S,X,U)$-measurable. 
\end{cor} 

\begin{proof} First consider the case when $\Omega_X$ is an uncountably infinite set. 
Theorem \ref{thm:kallenberg} 
implies $X=f(S,R)$ almost surely for some measurable $f$ and some random variable $R \sim \script{U}[0,1]$ that is independent of $S$. 
Let $C$ be an uncountable Borel measurable subset of $[0,1]$ that has measure 0, such as the Cantor set. Let $b:C\rightarrow \Omega_X$ be an isomorphism (recall Lemma \ref{lem:uncountable-isomorphism}). 
Define the random variable $W:\Omega\rightarrow [0,1]$ by 
$$ W = \left\{\begin{array}{cc}
R & \mbox{ if $X=f(S,R)$ and $R\notin C$} \\
b^{-1}(X) & \mbox{ else} 
\end{array}\right.$$ 
Since $P[X=f(S,R)]=1$ and $P[R \notin C]=1$ we have that $P[W=R]=1$ and so $W$ is also uniformly distributed over $[0,1]$ and independent of $S$.  By definition of $W$ we have 
\begin{align}
&W \notin C \implies \mbox{($W=R$ and $X=f(S,R)$)}\label{eq:notin} \\
&W \in C \implies W=b^{-1}(X) \label{eq:in} 
\end{align} 
Define the measurable function $g:\Omega_S\times [0,1]\rightarrow \Omega_X$ by 
$$ g(s, w) = \left\{\begin{array}{cc}
f(s,w) & \mbox{ if $w \notin C$} \\
b(w) & \mbox{ else}
\end{array}\right.$$
It remains to show $X=g(S,W)$. If $W \notin C$ then by definition of $g$ we have 
\begin{align*}
g(S,W)&=f(S,W)\\
&\overset{(a)}{=}X
\end{align*}
where (a) holds by \eqref{eq:notin}. If $W \in C$ then by definition of $g$ we have 
$$g(S,W)=b(W)=b(b^{-1}(X)) = X$$ 
where we have used \eqref{eq:in}.  The case when $\Omega_X$ is finite or countably infinite is similar and proceeds by 
defining $C$ as a subset of $[0,1]$ with the same cardinality as $\Omega_X$. 
\end{proof} 

%The next corollary applies this to the stochastic control problem. 

\begin{cor} \label{cor:vk} If random elements $(S_k)_{k=1}^{\infty}$ and $(X_k)_{k=1}^{\infty}$  satisfy  $X_k \in C(S_k)$ for all $k \in \mathbb{N}$ (where each $X_k:\Omega\rightarrow \Omega_X$ is measurable with respect to the Borel space $(\Omega_X, \script{F}_X)$; each $S_k:\Omega\rightarrow \Omega_S$ is measurable with respect to the general measurable space $(\Omega_S, \script{F}_S)$), 
and 
if there is a random variable $U\sim \script{U}[0,1]$ that is 
independent of $(S_k, X_k)_{k=1}^{\infty}$, then 

a) For all $k \in \mathbb{N}$ we (surely) have 
$X_k = g_k(S_k, W_k) \in C(S_k)$
for some measurable function $g_k:\Omega_S\times [0,1]\rightarrow \Omega_X$ and some random variable
$W_k \sim \script{U}[0,1]$ that is independent of $S_k$. 

b) If Assumptions \ref{assumption:1}-\ref{assumption:2} hold then for all $k \in \mathbb{N}$ there is a measurable function 
$v_k:\Omega_S\times [0,1]\rightarrow \Omega_X$ that satisfies 
$v_k(s, r) \in C(s)$ for all $(s,r) \in \Omega_S \times [0,1]$ 
such that 
\begin{equation} \label{eq:Xv2} 
X_k=v_k(S_k, W_k) 
\end{equation} 
where the random variables $W_k$ 
are the same as in part (a). 
\end{cor} 

\begin{proof} 
Part (a) follows immediately from Corollary \ref{cor:sure-kallenberg}.  
To prove (b), fix $k \in \{1, 2, 3, \ldots\}$ and define
$$ v_k(s,r) = \left\{\begin{array}{cc}
g_k(s,r) & \mbox{ if $g_k(s,r) \in C(s)$} \\
\psi(s) & \mbox{ else} 
\end{array}\right.$$ 
where $g_k$ is from part (a). 
Assumptions \ref{assumption:1} and \ref{assumption:2} and measurability of $g_k$ ensure measurability of $v_k$. 
Since $\psi(s)\in C(s)$ for all $s \in \Omega_S$, it is clear that $v_k(s,r) \in C(s)$ for all $(s,r) \in \Omega_S\times [0,1]$. 
By part (a) it holds that $X_k=v_k(S_k,W_k)$. 
\end{proof} 

The equality \eqref{eq:Xv2} has a simpler structure than \eqref{eq:Xv}. However, the $v_k$ functions in \eqref{eq:Xv} completely specify a causal and measurable control policy.  In contrast, the $v_k$ functions in \eqref{eq:Xv2} do not specify a control policy because each $W_k$ may have some unknown dependence on $S_1, \ldots, S_{k-1}$ as well as on additional sources of (potentially noncausal) randomness. 
%The Corollary \ref{cor:vk} does not require the causality constraint \eqref{eq:control1} and so $W_k$ can 
%depend on values $S_{k+1}, S_{k+2}, S_{k+3}, \ldots$ from the future. 

\subsection{Opportunistic scheduling} 

The following special case  is of interest in the area of wireless networks.  
Fix $m \in \mathbb{N}$ and let $(\mathbb{R}^m, \script{B}(\mathbb{R}^m))$ be the measurable space for the decision elements $X_k$. There are $m$ different wireless links that can change over time according to states $(S_k)_{k=1}^{\infty}$, where $S_k$ describes the state of all channels on slot $k$. At the start of each slot $k \in \mathbb{N}$ we observe $S_k$ and then choose a   transmission rate vector $X_k \in C(S_k)$, where $C(S_k)\subseteq \mathbb{R}^m$ is the set of all transmission rate options available when the channel state is $S_k$ (different rate options arise, for example, from different modulation and coding choices).   This is called an \emph{opportunistic scheduling system} because the state $S_k$ is known before  $X_k$ is decided. Control strategies for 
such systems consider network 
stability \cite{tass-server-allocation}\cite{neely-power-network-jsac}, 
utility maximization \cite{now}\cite{sno-text}\cite{shroff-opportunistic}\cite{atilla-primal-dual-jsac}\cite{stolyar-greedy}\cite{neely-fairness-ton}, and energy minimization \cite{neely-energy-it}.  
Assume  $(S_k)_{k=1}^{\infty}$ are identically distributed random elements associated with a measurable space $(\Omega_S, \script{F}_S)$ and  a distribution 
$\lambda:\script{F}_S\rightarrow [0,1]$: 
$$ \lambda(A) = P[S_k \in A] \quad \forall A \in \script{F}_S$$
The full sequence $(S_k)_{k=1}^{\infty}$ is ``chosen by nature'' at time 0.  In a causal decision scenario, on step $k$ the controller only knows the values of $S_1, \ldots, S_k$ before choosing $X_k \in C(S_k)$.  In a noncausal scenario the full $(S_k)_{k=1}^{\infty}$ sequence is known. 
Assume   $\script{G}\subseteq \script{F}$ is a subsigma algebra independent of $\sigma((S_k)_{k=1}^{\infty})$ that is used as a 
source of randomness to facilitate stochastic decisions.  Assume there is a random variable $U \sim \script{U}[0,1]$ that is $\script{G}$-measurable. 
 
The work \cite{sno-text} defines the  
\emph{network rate region} $\Gamma$ as the set of all expectations of $X_1$ that can be achieved on the first slot, shows this set is the same for all slots, and determines the fundamental \emph{capacity region} (see also \cite{tass-server-allocation}\cite{neely-power-network-jsac}) when such a transmission system is used for single and multi-hop queueing networks.\footnote{For 1-hop networks the \emph{capacity region} is the set of all vectors that are dominated by a vector in the closure of $\Gamma$. For multi-hop networks the capacity region depends on all possible multi-hop flow allocations available on graphs associated with points in the closure of $\Gamma$ \cite{neely-power-network-jsac}\cite{now}\cite{tass-radio-nets}.}  The argument in \cite{sno-text} implicitly allows expanding the probability space to ensure the sigma algebra $\script{G}$ is complex enough to emulate an independent virtual system with identical stochastics over any number of virtual slots before the slot $1$ decision on the actual system is made.  The next results  do not require expanding the probability space and allow $\script{G}$ to be as simple as $\script{G}=\sigma(U)$. 

%For simplicity, assume all $C(S_k)$ sets are bounded.

\begin{assumption} \label{assumption:3}  For the  function $C:\Omega_S\rightarrow Pow(\mathbb{R}^m)$, there is a bounded subset $D\subseteq \mathbb{R}^m$ such that $C(s)$ is nonempty and  $C(s) \subseteq D$ for all $s \in \Omega_S$.\footnote{Assumption \ref{assumption:3} is mainly for convenience and can be replaced by the weaker  assumption that expectations of random vectors $X_k \in C(S_k)$ are finite.}
\end{assumption}

\begin{defn}  Given a distribution $\lambda:\script{F}_S\rightarrow [0,1]$ and a function $C:\Omega_S\rightarrow Pow(\mathbb{R}^m)$ that satisfies Assumption \ref{assumption:3}, define the \emph{rate region} $\Gamma \subseteq \mathbb{R}^m$ as the set of all expectation vectors 
$\expect{v(S,U)}$ that can be achieved by some measurable function $v:\Omega_S\times \mathbb{R}\rightarrow \mathbb{R}^m$ that satisfies $v(s,w) \in C(s)$ for all $s\in \Omega_S$ and $w \in \mathbb{R}$, and on a probability space with independent random elements $S$ and $U$ such that $S$ has distribution $\lambda$ and $U\sim\script{U}[0,1]$.  
\end{defn} 

%If 
%Assumption \ref{assumption:1} holds, we can guarantee $\Gamma$ is a nonempty set by 
%considering the measurable function $v(s,w) = \psi(s)$.  
Define $\overline{\Gamma}$ as the closure of the set $\Gamma$. 
Using Corollary \ref{cor:vk}b, it is straightforward to show that Assumptions \ref{assumption:1},  \ref{assumption:2}, \ref{assumption:3} imply that $\Gamma$ is nonempty, bounded, and convex, while $\overline{\Gamma}$ is nonempty, compact, and convex.  It can be shown the definition of $\Gamma$ is unchanged if one allows $U$ to be a  
random variable of \emph{any} distribution, provided that $U$ and $S$ are independent. 
The next lemma shows that  $\overline{\Gamma}$ captures all 
time average expectations of $X_k$ that can be achieved at any time $k$ by a measurable decision policy  for choosing $X_k\in C(S_k)$,  regardless of whether or not the policy is causal.  Sample path time averages are also considered in the lemma using a theory of \emph{constructible sets} \cite{constructible-book}.  Counter-examples in Section \ref{section:counter-example} show that time averages can be far outside the set $\overline{\Gamma}$ if the controller can make nonmeasurable decisions.

\begin{prop}  Suppose Assumptions \ref{assumption:1}, \ref{assumption:2}, \ref{assumption:3} hold for the opportunistic scheduling problem with identically distributed random elements $(S_k)_{k=1}^{\infty}$ with some distribution $\lambda$.   Let 
$(X_k)_{k=1}^{\infty}$ be a sequence of (Borel measurable) random vectors that satisfy  
$X_k\in C(S_k)$ surely for each $k \in \mathbb{N}$. Then 

a) For all $k \in \mathbb{N}$ we have $\expect{X_k} \in \Gamma$  and $\frac{1}{k}\sum_{i=1}^k \expect{X_i} \in \Gamma$.  
%Hence, any convergent subsequence of $(\frac{1}{k}\sum_{i=1}^k \expect{X_i})_{k=1}^{\infty}$ converges to a point in $\overline{\Gamma}$. 

b) If $(S_k)_{k=1}^{\infty}$ is i.i.d. and $S_k$ is independent of $(X_1, \ldots, X_{k-1})$ 
for all $k \in \{2, 3, 4, \ldots\}$ then for all $k\in \mathbb{N}$
$$ \expect{X_k|\script{H}_k} \in \overline{\Gamma} \quad \mbox{almost surely} $$
where $\script{H}_k=\sigma(X_1, \ldots, X_{k-1})$ for $k\geq 2$ and $\script{H}_1=\{\phi, \Omega\}$. 

c) If $(S_k)_{k=1}^{\infty}$ is i.i.d. and $S_k$ is independent of $(X_1, \ldots, X_{k-1})$ for all $k \in \{2, 3, 4, \ldots\}$ then  
$$ \mbox{$\lim_{k\rightarrow\infty} \mbox{dist}\left(\frac{1}{k}\sum_{i=1}^k X_i, \overline{\Gamma}\right) = 0 \quad \mbox{almost surely}$}$$
where $\mbox{dist}(x, \overline{\Gamma})$ is the Euclidean distance between a point $x \in \mathbb{R}^m$ and the compact and convex set $\overline{\Gamma} \subseteq \mathbb{R}^m$. 
 \end{prop} 

\begin{proof} Without loss of generality, for parts (a)-(b) we can assume existence of a random variable 
$U\sim \script{U}[0,1]$ of the form $U:\Omega\rightarrow [0,1]$ that is  independent of $(S_k, X_k)_{k=1}^{\infty}$. Indeed, if this does not hold then we can extend the probability space to a new space
$(\tilde{\Omega}, \tilde{\script{F}}, \tilde{P})$ such that 
$$\tilde{\Omega} = \Omega \times [0,1], \quad \tilde{\script{F}} = \script{F} \otimes \script{B}([0,1]), \quad \tilde{P}=P \otimes \mu$$
where $\mu$ is the standard Borel measure on Borel subsets of $[0,1]$.  Each outcome of the new sample space has the form 
$\tilde{\omega} = (\omega, t)$ where $\omega \in \Omega$ and $t \in [0,1]$.    Then define 
$\tilde{S}_k:\tilde{\Omega}\rightarrow \Omega_S$ and $\tilde{X}_k:\tilde{\Omega}\rightarrow \mathbb{R}^m$ by
$$
\tilde{S}_k(\omega, t) = S_k(\omega) \quad , \quad \tilde{X}_k(\omega, t) = X_k(\omega)
$$
Also define $U:\tilde{\Omega} \rightarrow [0,1]$ by $U(\omega, t) = t$ and observe that $U\sim \script{U}[0,1]$ and $U$ is independent of $(\tilde{X}_k,\tilde{S}_k)_{k=1}^{\infty}$. Then $(\tilde{X}_k,\tilde{S}_k)_{k=1}^{\infty}$ on the extended probability space has the same distribution 
as $(X_k, S_k)_{k=1}^{\infty}$ on the original space. Thus, $X_k$ and $\tilde{X}_k$ have the same expectation (useful for part (a));  $f(X_1, \ldots, X_{k-1})$ and $f(\tilde{X}_1, \ldots, \tilde{X}_{k-1})$ have the same distribution for any measurable function $f$ (useful for part (b)). 

To prove (a), suppose there is a sequence of random vectors
$(X_k)_{k=1}^{\infty}$ that satisfy  $X_k\in C(S_k)$ surely for each $k \in \{1, 2, 3, \ldots\}$. 
Assuming existence of $U\sim\script{U}[0,1]$ that is independent of $(S_k, X_k)_{k=1}^{\infty}$, apply Corollary \ref{cor:vk}b to obtain 
\begin{equation} \label{eq:X-tilde}
X_k=v_k(S_k, W_k)
\end{equation} 
for some measurable function $v_k:\Omega_S\times [0,1]\rightarrow \mathbb{R}^m$ that satisfies $v_k(s,v) \in C(s)$ for all $s\in \Omega_S, t \in [0,1]$ and some random variable $W_k \sim \script{U}[0,1]$ that is independent of $S_k$. By definition of $\Gamma$ it holds that $\expect{X_k} \in \Gamma$.  This holds for all $k \in \mathbb{N}$. Convexity of $\Gamma$ ensures that $\frac{1}{k}\sum_{i=1}^k \expect{X_k} \in \Gamma$ for all $k \in \mathbb{N}$.

To prove (b), assume $(S_k)_{k=1}^{\infty}$ are i.i.d. and fix $k \in \{2, 3, 4, \ldots\}$.  
Let $Z$ be a version of $\expect{X_k|\script{H}_k}$. Since $\overline{\Gamma}$ is a closed subset of $\mathbb{R}^m$ it is Borel measurable and $\{Z \notin \overline{\Gamma}\}$ is an event.  Suppose $P[Z \notin  \overline{\Gamma}]>0$ (we reach a contradiction). 
Since $\overline{\Gamma}$ is compact and convex it is \emph{constructible}, meaning
it is the countable intersection 
of closed half-spaces (see Proposition 7.5.6 in 
\cite{constructible-book}): 
\begin{equation} \label{eq:countable-intersection}
 \overline{\Gamma} = \cap_{j=1}^{\infty} \{x \in \mathbb{R}^m : a_j^{\transpose}x \leq b_j\}
 \end{equation} 
 for some  $a_j \in \mathbb{R}^m$ and $b_j \in \mathbb{R}$ for $j \in \mathbb{N}$.  Then 
 $$P\left[Z \notin \overline{\Gamma}\right] = P[\cup_{j=1}^{\infty} \{a_j^{\transpose}Z > b_j\}]$$
Since the above probability is positive, there must be an index $i \in \mathbb{N}$ for which 
$P[a_i^{\transpose}Z>b_i]>0$.  Define the event $A = \{a_i^{\transpose}Z>b_i\}$. Then $P[A]>0$.  
Since $Z$ is a version of $\expect{X_k|\script{H}_k}$ it is $\script{H}_k$-measurable and $A \in \script{H}_k$. By the defining property of a conditional expectation it holds that 
$$ \expect{Z1_A} = \expect{X_k1_A}$$
and so 
$$ \frac{\expect{a_i^{\transpose}Z1_A}}{P[A]} = \frac{a_i^{\transpose}\expect{X_k1_A}}{P[A]} $$
Since $U \sim \script{U}[0,1]$ is independent of the random element $(S_k, 1_A) \in \Omega_S \times \{0,1\}$ we have by Corollary \ref{cor:vk}b 
$$ X_k = v(S_k, 1_A, W)$$
for some random variable $W \sim \script{U}[0,1]$ that is independent of $(S_k, 1_A)$, and 
for some measurable 
function $v:\Omega_S\times \{0,1\} \times [0,1]\rightarrow \mathbb{R}^m$ such that 
$$v(s, b, w) \in C(s) \quad \forall s \in \Omega_S, b \in \{0,1\}, w \in [0,1]$$
Then 
\begin{align} 
\frac{\expect{a_i^{\transpose}Z1_A}}{P[A]} &= \frac{a_i^{\transpose}\expect{v(S_k, 1_A, W)1_A}}{P[A]}\\
&=a_i^{\transpose}\expect{v(S_k, 1, W)|A}  \\
&=a_i^{\transpose} \expect{v(S_k, 1,  W)}
\end{align} 
where we have used the fact that $(S_k, W)$ is independent of $A$ (recall that $W$ is independent of $(S_k, 1_A)$ and $S_k$ is independent of the $\script{H}_k$-measurable random variable $1_A$, so that $S_k, W, 1_A$ are mutually independent).   Defining $y=\expect{v(S_k,1,W)}$, it follows by 
definition of $\Gamma$ that $y \in \Gamma \subseteq \overline{\Gamma}$ and so $a_i^{\transpose}y\leq b_i$ (recall \eqref{eq:countable-intersection}).   Thus
$$\frac{\expect{a_i^{\transpose}Z1_A}}{P[A]} \leq b_i$$
and so 
\begin{equation}\label{eq:negative} 
\expect{(a_i^{\transpose}Z - b_i)1_A} \leq 0
\end{equation} 
By definition of event $A=\{a_i^{\transpose}Z>b_i\}$, the random variable $(a_i^{\transpose}Z-b_i)1_A$ is nonnegative, so  \eqref{eq:negative} implies
$$(a_i^{\transpose}Z - b_i)1_A = 0 \quad \mbox{almost surely} $$
However, $(a_i^{\transpose}Z - b_i)1_A>0$ if and only if $1_A>0$, which contradicts the fact that $P[A]>0$. 

To prove (c), define $M_1=X_1-\expect{X_1}$ and define 
$M_k=X_k - \expect{X_k|\script{H}_k}$ for $k \in \{2, 3, 4, \ldots\}$. Then $\{M_k\}_{k=1}^{\infty}$ 
is a zero mean sequence of bounded random vectors in $\mathbb{R}^m$ where each component forms a martingale difference, and so the law of large numbers for martingale differences implies that $\frac{1}{k}\sum_{i=1}^kM_i$ converges to 0 almost surely \cite{chow-lln}. Thus, $\frac{1}{k}\sum_{i=1}^kX_i - \frac{1}{k}\sum_{i=1}^k\expect{X_k|\script{H}_k}$ converges to 0 almost surely. However, part (b) implies $\expect{X_k|\script{H}_k}\in \overline{\Gamma}$ almost surely, and convexity of $\overline{\Gamma}$
implies $\frac{1}{k}\sum_{i=1}^k\expect{X_k|\script{H}_k} \in \overline{\Gamma}$ almost surely. 
\end{proof} 

Conversely, if $(S_k)_{k=1}^{\infty}$ is identically distributed with distribution $\lambda$ and $U\sim \script{U}[0,1]$ is independent of $(S_k)_{k=1}^{\infty}$, then for any $x \in \overline{\Gamma}$ it is straightforward to see there is a causal and measurable decision policy that chooses $X_k \in C(S_k)$ to ensure: 
\begin{align*}
\mbox{$\lim_{k\rightarrow\infty} \frac{1}{k}\sum_{i=1}^k \expect{X_i} = x$}
\end{align*}
and if $(S_k)_{k=1}^{\infty}$ are i.i.d. then we can further ensure
\begin{align*}
\mbox{$\lim_{k\rightarrow\infty} \frac{1}{k}\sum_{i=1}^k X_i = x \quad \mbox{almost surely}$}
\end{align*}
This is done by first mapping the random variable $U \sim \script{U}[0,1]$ to a sequence of i.i.d. $\script{U}[0,1]$ random variables $(U_1, U_2, U_3, \ldots)$ by using the bits corresponding to the binary expansion of $U$.  Then fix any sequence of points $x_k \in \Gamma$ that satisfy $x_k\rightarrow x$ and define $X_k=v_k(S_k, U_k)$ for $k \in \mathbb{N}$ where $v_k$ is the corresponding measurable function in the definition of $\Gamma$ that ensures $X_k \in C(S_k)$ and $\expect{X_k} = x_k$ for all $k \in \mathbb{N}$.

%\begin{proof}(Convexity)
%Let $x_1$ and $x_2$ be elements of $\Lambda$ and let $p \in (0,1)$.    Since $x_1$ and $x_2$ are in $\Lambda$ there are measurable functions $v_1$ and $v_2$ such that 
%\begin{align*}
%x_1 &= \expect{v_1(S,W)} \\[
%x_2 &= \expect{v_2(S,W)} 
%\end{align*}
%Define the measurable function $v:\Omega_S\times [0,1]\rightarrow \mathbb{R}^m$ by 
%$$ v(s,w) = \left\{\begin{array}{cc}
%v_1\left(s,\frac{w}{p}\right) & \mbox{ if $w\leq p$} \\
%v_2\left(s,\frac{(w-p)}{(1-p)}\right) & \mbox{ if $w>p$} 
%\end{array}\right.$$
%By definition of $\Lambda$ we have $\expect{v(S,W)} \in \Lambda$. However 
%\begin{align*}
%\expect{v(S,W)} &= p\expect{v_1\left(S,\frac{W}{p}\right) \vert W\leq p} + (1-p)\expect{v_2\left(S,\frac{(W-p)}{(1-p)}\right) \vert W>p}\\
%&= px_1 + (1-p)x_2
%\end{align*}
%which holds because: 
%\begin{itemize} 
%\item Given $W\leq p$,  the random variable $W/p$ is conditionally independent of $S$ and has conditional distribution equal to a uniform distribution over $[0,1]$. 
%\item Given $W>p$, the random variable $(W-p)/(1-p)$ is conditionally  independent of $S$ and has conditional distribution equal to a uniform distribution over $[0,1]$. 
%\end{itemize} 
%Thus $px_1 + (1-p)x_2 \in \Lambda$. 
%\end{proof} 

\section{Counter-examples} \label{section:counter-example} 

This section considers pathological cases for the opportunistic scheduling
problem with i.i.d. channel states $(S_k)_{k=1}^{\infty}$ and only $m=1$ channel.  
We use $(\Omega_S, \script{F}_S)=([0,1], \script{B}([0,1]))$ throughout. 

\subsection{Nonmeasurable policies} \label{section:nonmeasurable} 

This example gives an opportunistic scheduling 
system for which a nonmeasurable policy produces  
larger time averages in comparison to any measurable policy.  
It is similar in spirit to the example given by 
Blackwell in \cite{blackwell-memoryless} which shows that a single 
decision at one step of a finite stage dynamic program can bring arbitrarily more utility if it makes a measurable decision based on both the current state and memory, rather than only on the current state (even if the utility function depends only on the current state and the current decision). However, the structure of that example is different from ours and compares two measurable policies rather than a nonmeasurable policy in comparison to any measurable decision.

Define 
$$\Omega = [0,1]^{\mathbb{N}}, \quad \script{F}=\otimes_{k\in \mathbb{N}} \script{B}([0,1]), \quad P=\otimes_{k \in \mathbb{N}}\mu$$ 
where $\mu$ is the standard Borel measure on $[0,1]$. Each outcome has the structure 
$\omega = (\omega_0, \omega_1, \omega_2, \ldots)$. Define $U:\Omega\rightarrow [0,1]$ and $S_k:\Omega\rightarrow [0,1]$ by $U(\omega)=\omega_0$ and $S_k(\omega) = \omega_k$ for $k \in \{1, 2, 3, \ldots\}$. 
Then  $(S_k)_{k=1}^{\infty}$ are i.i.d. $\script{U}[0,1]$ variables; $U \sim \script{U}[0,1]$ is independent of $(S_k)_{k=1}^{\infty}$.

Fix $A\subseteq [0,1]$ as a set  with inner measure 0 and outer measure 1 (such sets exist under the Axiom of Choice \cite{sierpinski-strange}\cite{alexander-strange}\cite{strange-uniform}).  In particular, neither $A$ nor its complement $A^c = [0,1] \setminus A$ contains a Borel subset of positive measure. Define $C:[0,1]\rightarrow Pow(\mathbb{R})$ by 
$$ C(s)= \left\{\begin{array}{cc}
\{0,1\} & \mbox{ if $s \in A$} \\
\{0,2\} & \mbox{ if $s \in A^c$} 
\end{array}\right.$$
A measurable choice function for this system is $\psi(s)=0$ for all $s \in [0,1]$   and so Assumption \ref{assumption:1} holds. 
However, it can be shown that Assumption \ref{assumption:2} fails. 

Consider any decision policy that is measurable, meaning that it 
produces valid random variables $(X_k)_{k=1}^{\infty}$ that satisfy $X_k \in C(S_k)$ surely for all $k \in \mathbb{N}$. For each fixed $k \in \mathbb{N}$ define the set
\begin{align*}
D_k &= \{\omega \in \Omega : \omega_k \in A\}
\end{align*}
It turns out that  every $\script{F}$-measurable subset of $D_k$ has measure 0, and every $\script{F}$-measurable subset of $D_k^c$ also has measure 0 (proof postponed to the next paragraph). Since $X_k$ is a valid random variable we know $\{X_k=1\}$ and $\{X_k=2\}$ are valid events (i.e., in $\script{F}$). Since $X_k\in C(S_k)$ we have 
$\{X_k=1\} \subseteq D_k$ and so $P[X_k=1]=0$. Likewise, $\{X_k=2\} \subseteq D_k^c$ and so $P[X_k=2]=0$.  It follows that $X_k=0$ almost surely for all $k \in \mathbb{N}$ and so $\lim_{k\rightarrow\infty} \frac{1}{k}\sum_{i=1}^k X_i = 0$ almost surely.   On the other hand, the (nonmeasurable) policy that chooses $X_k=1$ if $S_k\in A$ and $X_k=2$ otherwise surely yields $X_k \in C(S_k)$ and $X_k\geq 1$ for all $k \in \mathbb{N}$, so $\liminf_{k\rightarrow\infty} \frac{1}{k}\sum_{i=1}^kX_i \geq 1$ surely. 

It remains to show that every $\script{F}$-measurable subset of $D_k$ has probability 0 (the corresponding proof for $D_k^c$ is similar). Let $Z$ be a $\script{F}$-measurable subset of $D_k$.  Let $\pi_k(Z)$ be the projection of $Z$ onto the dimension $k$.  Then $\pi_k(Z) \subseteq A$. Since $Z \in \script{F}$, $Z$ can be viewed as a Borel measurable subset of the Polish space $[0,1]^{\mathbb{N}}$, and since the projection onto one dimension is a continuous function, the set $\pi_k(Z)$ is an analytic subset of $[0,1]$ and hence a Lebesgue measurable subset of $[0,1]$ \cite{srivastava-borel}. It follows that $\pi_k(Z)=B \cup R$ where $B$ is a Borel set and $R$ is a subset of a Borel set $\tilde{R}$ with $\mu(\tilde{R})=0$. Since $B \subseteq \pi_k(Z)\subseteq A$ we have $\mu(B)=0$ (recall that all Borel measurable subsets of $A$ have measure 0).  Therefore
\begin{align*}
P[Z] &\leq P\left[\omega \in \Omega : \omega_k \in B \cup \tilde{R}\right] \\
&= \mu(B \cup \tilde{R}) \\
&\leq \mu(B) + \mu(\tilde{R})=0
\end{align*}

\subsection{Randomization without deterministic measurable choice} 

This example shows a probability space with 
random elements $X_k \in C(S_k)$ for all $k \in \mathbb{N}$, so that a form of randomized choice exists, without the measurable choice assumption (Assumption \ref{assumption:1}).  
Famous examples in the field of descriptive set theory by Blackwell \cite{blackwell-borel-not-containing-graph},  Novikoff \cite{novikoff-borel-example}, Sierpi{\'n}ski  \cite{sierpinski-borel-example}, and Addison \cite{addison-borel-example} prove existence of a Borel measurable set $A \subseteq [0,1]^2$ with a projection $\pi_1(A)$ onto the first component that satisfies $\pi_1(A)=[0,1]$ and  such that there is no measurable   function $\psi:[0,1]\rightarrow [0,1]$ that satisfies $(x, \psi(x)) \in A$ for all $x \in [0,1]$ (see also Example 5.1.7 in \cite{srivastava-borel}). Define  
$$C(s) = \{y: (s,y) \in A\} \quad \forall s \in [0,1]$$
Assumption \ref{assumption:2} holds for this system because $\{(s, x) \in [0,1]^2 :  x \in C(s)\} = A$ is a Borel set, while Assumption \ref{assumption:1} fails because there is no measurable choice function $\psi(s)$.  Nevertheless a probability space can have i.i.d. random vectors $(S_k, X_k)_{k=1}^{\infty}$ that satisfy $X_k \in C(S_k)$ surely for all $k \in \mathbb{N}$ as follows: Let $(U_k)_{k=0}^{\infty}$ be i.i.d. $\script{U}[0,1]$ variables.  Since $A$ is an uncountably infinite Borel measurable subset of $\mathbb{R}^2$, there is an isomorphism $b:[0,1]\rightarrow A$.  Define $(S_k, X_k) = b(U_k)$ for all $k \in \mathbb{N}$.  Then  $(S_k, X_k)_{k=1}^{\infty}$ are indeed i.i.d. random vectors and $(S_k, X_k) \in A$  and so $X_k \in C(S_k)$ surely for all $k \in \mathbb{N}$.  Fix $k \in \mathbb{N}$. Using randomization variable $U_0$ with Corollary \ref{cor:sure-kallenberg} implies $X_k=h_k(S_k, W_k)$ for some Borel measurable function $h_k$ and some random variable $W_k\sim \script{U}[0,1]$ that is independent of $S_k$.  Since Assumption \ref{assumption:2} holds, any probability space that contains a random element $(S, W)$ with the same distribution as $(S_k, W_k)$ yields  
$$ P[h_k(S,W)\in C(S)]=1$$

\section{Conclusion} 

This paper considers sequences of Borel measurable functions 
where each function $X_k$ 
is constrained to be measurable with respect to the sigma algebra generated
by the union of an arbitrary number of sigma 
algebras associated
with index $k$. Specifically, each $X_k$ is $\sigma(\cup_{j \in J_k}\script{H}_j)$-measurable, 
where $J_k$ is an arbitrary index set and $\script{H}_j$ is a sigma algebra for each $j \in J_k$.
It is shown that each $X_k$ can be expressed
as the composition of a Borel function $h_k$ with some 
real-valued measurable functions $Y_j$ for $j \in J_k$, each $Y_j$ being
measurable with respect to the sigma algebra $\script{H}_j$.  
The same $Y_j$ functions can be used to 
represent the influence of $\script{H}_j$ for all indices $k$ in which $j \in J_k$. 
For applications to stochastic control, this enables functional representations
of all possible decision vectors that satisfy causality and measurability constraints.  This leads to 
a refined theorem on network capacity for opportunistic scheduling in time-varying wireless
networks.  The theorem uses two measurability assumptions, including an assumption 
on existence of a measurable choice function.  By utilizing classical pathological counter-examples in the field of descriptive set theory, an example opportunistic scheduling system is developed for which a nonmeasurable policy yields significantly better time averages in comparison to any measurable policy.

\section*{Acknowledgements} 

This work was supported in part by one or more of: NSF CCF-1718477, NSF SpecEES 1824418.

% ------------------------------------------------------------------------
%GATHER{Xbib.bib}   % For Gather Purpose Only
%GATHER{Paper.bbl}  % For Gather Purpose Only
\bibliographystyle{unsrt}
\bibliography{../../../../latex-mit/bibliography/refs}
\end{document}